\documentclass[a4paper,12pt]{article}
\usepackage{amsmath}
\usepackage{amsthm}
\usepackage{amssymb}
\usepackage{mathtools}
\usepackage{epsfig}
\usepackage{hyperref}
\usepackage{color}
\usepackage{cite}
\newtheorem{theo}{Theorem}
\newtheorem{lemma}[theo]{Lemma}

\newtheorem{cor}[theo]{Corollary}

\newtheorem*{problem*}{Problem}

\newcommand{\NN}{\mathbb{N}}

\newcommand{\LL}{{\mathcal L}}

\begin{document}
\title{Fundamental cycles in grid graphs\thanks{The first, second and fourth authors were supported by the Alexander von Humboldt Foundation in the framework of the Alexander von Humboldt Professorship of the second author endowed by the Federal Ministry of Education and Research.}}
\author{Bart\l{}omiej Kielak\thanks{Institute of Mathematics, Leipzig University, Augustusplatz 10, 04109 Leipzig, Germany. E-mail: {\tt bartlomiej.kielak@uni-leipzig.de}.}\and
        \newcounter{lth}
        \setcounter{lth}{2}
        Daniel Kr\'al'\thanks{Institute of Mathematics, Leipzig University, Augustusplatz 10, 04109 Leipzig, and Max Planck Institute for Mathematics in the Sciences, Inselstra{\ss}e 22, 04103 Leipzig, Germany. E-mail: {\tt daniel.kral@uni-leipzig.de}.}\and
	Ander Lamaison\thanks{Universidad P\'ublica de Navarra, Edificio Las Encinas, Campus de Arrosad\'ia, 31006 Pamplona, Spain. E-mail: {\tt ander.lamaison@unavarra.es}.}\and
        Xichao Shu$^\fnsymbol{lth}$}

\date{}
\maketitle

\begin{abstract}
We show that the average length of a fundamental cycle with respect to any fixed spanning tree of the $n\times n$ square grid
is at least $\Omega(\log n)$; the bound is asymptotically tight.
This result answers in the affirmative a question posed by McCarty in relation to sparse representations of binary matroids.
\end{abstract}

\section{Introduction}
\label{sec:intro}

Sparsity is an important concept related to the design of efficient algorithms.
In the graph setting,
the recently emerged notions of classes with bounded expansion and nowhere-dense classes of graphs
provided a unified view~\cite{DvoKT10,DvoKT13,GroKS14,GroKS17} on classical algorithmic results such as
Courcelle's MSO property testing for graphs with bounded tree-width~\cite{Cou90} and
other meta-algorithmic results for sparse graphs~\cite{CouMR00,FriG99,FriG01,See96}.
We refer to the monograph by Ne{\v{s}}etril and Ossona de Mendez~\cite{NesO12} for a comprehensive treatment of the two notions.
In the matroid setting, the notion of branch-width
led to similarly impactful algorithmic results~\cite{GavKO12,Hli03a,Hli03b,Hli06,HliO07,HliO08,JeoKO18}.

Sparse representations of matroids turned out to be of importance in relation to combinatorial optimization.
Empirical results on efficient algorithms for integer programs possessing a block structure~\cite{BorFM98,BerCCFLMT15,KhaEE18,FerH98,AykPC04,WeiK71,WanR13,GamL10,VanW10}
were complemented by theoretical results on $2$-stage integer programs due to Hemmecke and Schultz~\cite{HemS03},
which were further investigated in~\cite{AscH07,Kle21,CslEPVW21,JanKL21,KleR21,KouLO18}, and
$n$-fold integer programs introduced by De Loera et al.~\cite{DHOW} and
further studied in~\cite{HemOR13,CheM18,EisHK18,JanLR20,CslEHRW21,KouLO18}.

The above mentioned tractability results on integer programs with sparse constraint matrices
were unified and generalized using depth parameters of graphs derived from constraint matrices.
Ganian and Ordyniak~\cite{GanO18} showed that integer programs with bounded primal tree-depth of the constraint matrix and
bounded right side can be solved efficiently, and
Kouteck\'y, Levin and Onn~\cite{KouLO18} widely generalized this result to
integer programs with bounded primal or dual tree-depth of the constraint matrix
while dropping the condition on the right side.
Bria\'nski et al.~\cite{BriKKPS22,BriKKPS24} and Cooper et al.~\cite{ChaCKKP19,ChaCKKP20}
related the minimum possible primal and dual tree-depth of the constraint matrix of an integer program
equivalent to the input integer program to depth parameters of the column matroid of the constraint matrix (this
matroid is invariant under elementary row operations) and
provided algorithmic results yielding preconditioners for integer programming.

These developments sparked interest in the study of sparse representations of matroids,
which was also the subject of a blog post by McCarty~\cite{McC24}
entitled \emph{Sparse representations of binary matroids}.
Specifically, McCarty defined the \emph{sparsity} of a binary matroid $M$
as the minimum number of $1$'s in the binary representation of $M$.
There are two possible variants that one may consider:
the minimum can be taken over all representations of $M$ or
it can be taken only over those in the echelon form,
i.e.~those where one of the bases of $M$ is represented by unit vectors.
Observe that a graphic matroid has a representation in the echelon form such that
the average support of a vector in the representation is bounded
if and only if
there exists a spanning tree such that the average length of a fundamental cycle is bounded.
McCarty~\cite{McC24} suspected that the graphic matroids associated with grid graphs
did not enjoy this property.

\begin{problem*}[{McCarty~\cite[Problem 2]{McC24}}]
Is it true that for each integer $k$,
there exists an integer $n=n(k)$ so that if $T$ is any spanning tree of the $n\times n$ grid,
then the average length of a fundamental cycle with respect to $T$ is at least $k$?
\end{problem*}

We answer this problem in the affirmative by showing that $n(k)=2^{\Theta(k)}$.
Specifically,
we show that the average length of a fundamental cycle with respect to any spanning tree of the $n\times n$ grid
is at least $\frac{2}{625}\left\lfloor\log_5 n\right\rfloor$ (Corollary~\ref{cor:lower}).
We complement this result by constructing a spanning tree $T$ of an $n\times n$ such that
with the average length of a fundamental cycle with respect to $T$ at most $40\log_2 n$ (Corollary~\ref{cor:upper}).
We remark that in the arguments presented in this paper,
we focus on establishing the tight asymptotic behavior
without an attempt to optimize the multiplicative constants in our results.

\section{Notation}
\label{sec:not}

We now fix basic notation used throughout the paper.
Additional notation, which is used only in Section~\ref{sec:lower},
is introduced in the beginning of that section.

The set of the first $k$ positive integers is denoted by $[k]$.
A grid is the Cartesian product of two paths, and
the grid that is the product of two $n$-vertex paths is denoted by $G_{n,n}$ and
it is referred to as the \emph{$n$-grid}.
If $n$ is not important or clear from the context, we will simply say a grid.
The vertices of a grid with degree less than four are referred to as \emph{peripheral}.
We think of grids as embedded in the plane with faces being axis parallel unit squares;
in particular, peripheral vertices are exactly those incident with the outer face.
The paths of a grid that are horizontal in this embedding are referred to as \emph{rows} and
those that are vertical as \emph{columns}.

If $T$ is a spanning tree of a graph $G$ and $e$ is an edge of $G$ not contained in $T$,
then the \emph{fundamental cycle} associated with $e$ is the cycle formed by the edge $e$ and
the unique path of $T$ joining the two end vertices of $e$.
Finally,
we write $\LL(T,G)$ for the sum of the lengths of the fundamental cycles associated with all edges not contained in $T$.

\section{Upper bound}
\label{sec:upper}

In this section, we construct a spanning tree of a grid such that
the average length of a fundamental cycle is logarithmic in the size of the grid.

\begin{theo}
\label{thm:upper}
For every $n\in\NN$,
the $n$-grid $G_{n,n}$ has a spanning tree $T$ such that
$\LL(T,G_{n,n})\le 10n^2\log_2 n$.
\end{theo}

\begin{proof}
For every $n\in\NN$, we construct a spanning tree $T_n$ of the $n$-grid $G_{n,n}$
with the property given in the statement of the theorem.
The tree $T_1$ consists of the single vertex of $G_{1,1}$ and
the trees $T_2$ and $T_3$ are depicted in Figure~\ref{fig:upper23}.
Each tree $T_n$ will have the right most vertex in the bottom row distinguished as the root.

\begin{figure}
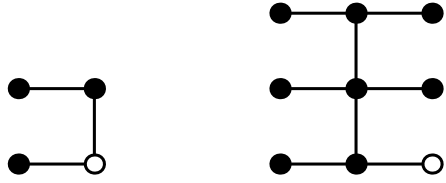

\begin{center}
\epsfbox{matgrid-1.mps}
\hskip 2cm
\epsfbox{matgrid-2.mps}
\end{center}
\caption{The spanning trees $T_2$ and $T_3$ constructed in the proof of Theorem~\ref{thm:upper};
         the positions of the vertices in the figure determine the trees.
         The root vertex is depicted by the empty circle and the edges not contained in the tree are omitted.}
\label{fig:upper23}
\end{figure}

We now present the recursive construction for $n\ge 4$.
The construction is illustrated in Figure~\ref{fig:upper-gen}.
If $n$ is odd, the tree $T_n$ consists of
the edges forming the bottom path of the grid,
the edges forming the central vertical path, and
four copies of $T_{(n-1)/2}$ tiling the rest in the grid.
The two copies of $T_{(n-1)/2}$ to the right from the horizontal path are flipped
so that their root vertices are adjacent to vertices of the central vertical path, and
the tree $T_n$ also contains the four edges joining the root vertices of the copies of $T_{(n-1)/2}$
to the neighboring vertices of the central vertical path.
The unique way of creating the tree $T_n$ in the described way is depicted in Figure~\ref{fig:upper-gen}.
If $n$ is even, the tree $T_n$ consists of
the edges forming the bottom path from its middle to the right,
the edges forming the central vertical path from the bottom to its middle,
three copies of $T_{n/2}$, one of them flipped, and a flipped copy of $T_{(n-2)/2}$ placed in a way that
the root vertex of the flipped copy of $T_{n/2}$ is on the central vertical path and
the root vertices of the remaining copies can be joined by an edge to a vertex on the central path.
Again, the unique way of creating the tree $T_n$ in the described way is depicted in Figure~\ref{fig:upper-gen}.
In addition to the general illustration provided in Figure~\ref{fig:upper-gen},
the trees $T_4$ and $T_5$ are depicted in Figure~\ref{fig:upper45}.

\begin{figure}
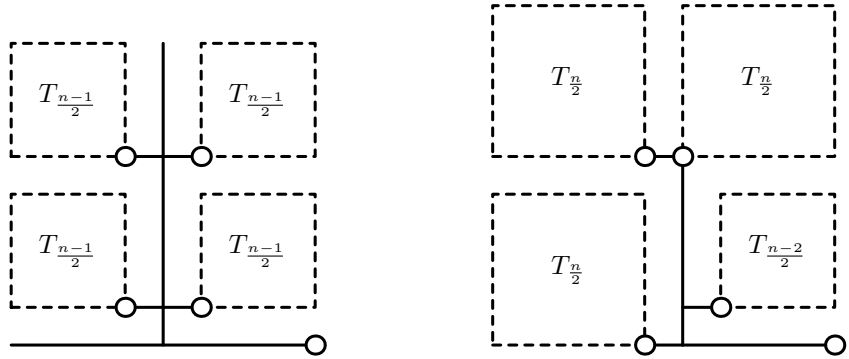

\begin{center}
\epsfbox{matgrid-3.mps}
\hskip 2cm
\epsfbox{matgrid-4.mps}
\end{center}
\caption{Illustration of the construction of the spanning tree $T_n$ when $n$ is odd (in the left) and when $n$ is even (in the right) for a general value of $n$ in the proof of Theorem~\ref{thm:upper}.}
\label{fig:upper-gen}
\end{figure}

\begin{figure}
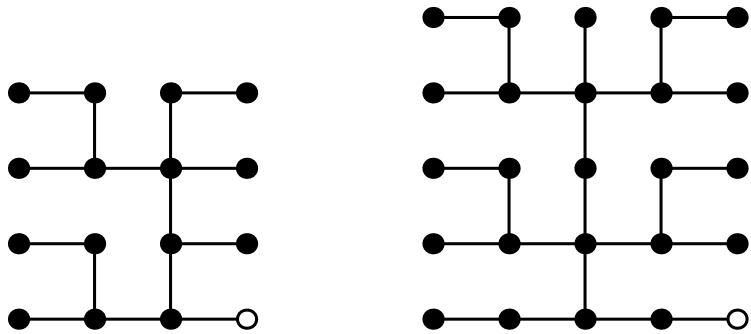

\begin{center}
\epsfbox{matgrid-5.mps}
\hskip 2cm
\epsfbox{matgrid-6.mps}
\end{center}
\caption{The spanning trees $T_4$ and $T_5$ constructed in the proof of Theorem~\ref{thm:upper}. 
         the positions of the vertices in the figure determine the trees.
         The root vertex is depicted by the empty circle and the edges not contained in the tree are omitted.}
\label{fig:upper45}
\end{figure}

Observe that the following holds for every $n\in\NN$:
the distance between the root vertex of $T_n$ and any vertex of $T_n$ is at most $2(n-1)$.
We can now estimate $\LL(T_n,G_{n,n})$.
A direct computation yields that $L_1=0$, $L_2=4$, $L_3=16$ and $L_4=33$ and
so the estimate claimed in the statement of the theorem holds when $n\in\{1,2,3,4\}$.
In the general case $n\ge 5$,
we can upper bound the value of $L_n$ as the sum of the lengths of fundamental cycles lying fully
within the recursive blocks $T_{\frac{n-2}{2}}$, $T_{\frac{n-1}{2}}$ and $T_{\frac{n}{2}}$ and
the sum of the lengths of fundamental cycles associated with edges that are not within the same block.
If $n$ is odd, there are exactly $4n-8$ such edges and
the length of the fundamental cycle associated with any of them is at most $\frac{5n-7}{2}$, and
so the sum of the lengths of fundamental cycles associated with edges that are not within the same block
is at most $10n^2$.
If $n$ is even, there are exactly $3n-6$ such edges and
the length of the fundamental cycle associated with any of them is at most $\frac{5n-2}{2}$, and
so the sum of the lengths of fundamental cycles associated with edges that are not within the same block
is at most $15n^2/2\le 10n^2$.

By induction, the sum of the lengths of fundamental cycles lying fully within the same recursive block is at most
\[10\left(\frac{n}{2}\right)^2\log_2\frac{n}{2},\]
we obtain that
\begin{align*}
\LL(T_n,G_{n,n}) & \le 4\cdot 10\left(\frac{n}{2}\right)^2\log_2\frac{n}{2}+10n^2 \\
                 & = 10n^2\log_2\frac{n}{2}+10n^2=10n^2\log_2 n.
\end{align*}
The proof of the theorem is now completed.
\end{proof}

Theorem~\ref{thm:upper} yields the following:

\begin{cor}
\label{cor:upper}
For every $n\ge 2$,
the $n$-grid $G_{n,n}$ has a spanning tree $T$ such that
the average length of a fundamental cycle with respect to $T$ is at most $40\log_2 n$.
\end{cor}

\begin{proof}
Observe that the number of fundamental cycles in the $n$-grid $G_{n,n}$ is equal to $2n(n-1)-(n^2-1)=(n-1)^2$.
Since it holds that $(n-1)^2\ge n^2/4$ for every $n\ge 2$,
the bound claimed in Corollary~\ref{cor:upper} now follows from Theorem~\ref{thm:upper}.
\end{proof}

\section{Lower bound}
\label{sec:lower}

Before presenting the proof of the main result, we need to introduce additional terminology.
Let $G$ be an $n$-grid and $C$ a cycle of $G$.
Recall that we think of $G$ as embedded in the plane with faces being axis parallel unit squares.
The \emph{width} of the cycle $C$ is the difference between the maximal and minimal $x$-coordinates of the vertices of $C$,
the \emph{height} is the difference between the maximal and minimal $y$-coordinates of the vertices of $C$,
the \emph{perimeter} is equal to sum of twice the width and twice the height;
in other words, the perimeter of $C$ is the perimeter of the smallest axis-parallel box bounding $C$.
Observe that the perimeter of $C$ is a lower bound on the length of $C$,
i.e.~the length of $C$ is at least its perimeter.

An \emph{expanded $n$-grid} is a graph $H$ obtained from the $n$-grid $G_{n,n}$ by duplicating some of its peripheral vertices (possibly the same vertex multiple times) and
adding some edges between duplicated vertices and peripheral vertices (adding an edge joining two duplicated vertices or
two peripheral vertices, including creating parallel edges, is allowed) so that the following holds:
for every $\delta>0$,
there exists a planar embedding of $H$ that can be obtained from $G_{n,n}$ by drawing the duplicated vertices and the added edges
in a way that every duplicated vertex has distance at most $\delta$ from its base vertex (the \emph{base vertex} of $v$
is the vertex of that $v$ is a duplicate).
The grid $G_{n,n}$ is referred to as the \emph{host grid} of $H$.
An example of an expanded $5$-grid can be found in Figure~\ref{fig:e5grid}.
Note that a particular case of an expanded $n$-grid is the $n$-grid itself.

\begin{figure}
\begin{center}
\epsfbox{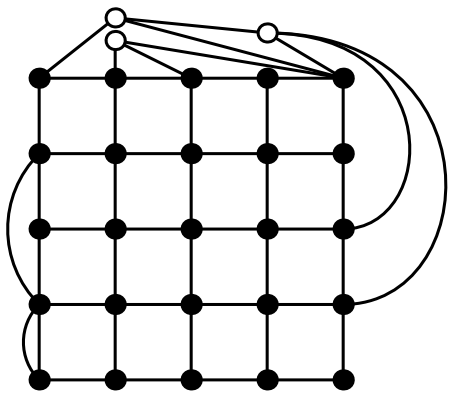}
\end{center}
\caption{An expanded $5$-grid. The duplicate vertices are drawn with empty circles.}
\label{fig:e5grid}
\end{figure}

In an expanded grid, each edge has a length that is defined as follows.
All edges of the host grid have a unit length, and
the length of an edge between two peripheral vertices $v$ and $w$ that is not an edge of the host grid
is the length of the shortest path between $v$ and $w$ using peripheral vertices only.
The length of an edge between a duplicated vertex $v$ and a peripheral vertex $w$
is the length of the shortest path from the base vertex of $v$ to the vertex $w$ using peripheral vertices only, and
the length of an edge between a duplicated vertex $v$ and another duplicated vertex $w$
is the length of the shortest path between the base vertices of $v$ and $w$ using peripheral vertices only.
We use these edge lengths to define the length of paths and cycles in an expanded grid.
The \emph{perimeter} of a cycle $C$ in expanded grid is the perimeter of the smallest axis-parallel box bounding
the vertices of the host grid contained in $C$ and the base vertices of the duplicated vertices contained in $C$.
Observe that the length of $C$ is always at least its perimeter (the edge lengths defined as above are required here).
Finally, we define $\LL^*(T,H)$ for a spanning tree $T$ of an expanded grid $H$
to be the sum of the perimeters of fundamental cycles associated with the edges of the host grid of $H$ that are not contained in $T$.

\begin{figure}
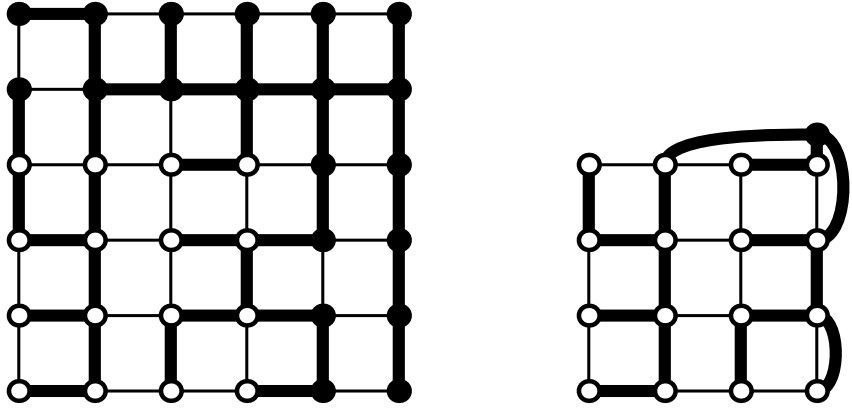

\begin{center}
\epsfbox{matgrid-8.mps}
\hskip 2cm
\epsfbox{matgrid-9.mps}
\end{center}
\caption{A $6$-grid $H$ with a spanning tree $T$ (drawn in the left) and
         the expanded grid $H[G'/T]$ (drawn in the right)
	 where $G'$ is the subgrid with the vertices drawn by empty circles.
	 The edges of the spanning tree $T$ and
	 the edges of the resulting spanning tree in $H[G'/T]$ are drawn in bold.}
\label{fig:HGT}
\end{figure}

Let $H$ be an expanded grid, $T$ a spanning tree of $H$ and $G'$ a subgrid of the host grid of $H$.
We define the expanded grid $H[G'/T]$ as follows.
Let $T'$ be the inclusion-wise minimal subtree of $T$ that contains all vertices of $G'$ and
suppress all vertices of degree two in $T'$ that are not contained in $G'$.
Each branching vertex $v$ of $T'$ that is not contained in $G'$
is considered to be a duplicate of the (unique) peripheral vertex of $G'$ that
is closest to the vertex $v$ if $v$ is a vertex of the host grid of $H$ or
that is closest to the base vertex of $v$ if $v$ is not a vertex of the host grid of $H$.
See Figure~\ref{fig:HGT} for illustration.
Note that the edges of $T'$ that are not edges of the host grid of $H$
join peripheral vertices of $G'$ and duplicated vertices only, and
the tree $T'$ is a spanning tree of $H[G'/T]$.
Also observe that planar embeddings of $H$ yield planar embeddings of $H[G'/T]$ as
in the definition of an expanded grid;
in particular, $H[G'/T]$ is an expanded grid with well-defined duplicate vertices.
Finally, observe if $e$ is an edge of the grid $G'$ that is not contained in the spanning tree $T'$,
then the perimeter of the fundamental cycle associated with $e$ in $T'$ (when the distances are measured in $H[G'/T]$)
is at most the perimeter of the fundamental cycle associated with $e$ in $T$ (when the distances are measured in $H$).

The following lemma contains the core argument giving the desired lower bound.

\begin{lemma}
\label{lem:lower}
Let $H$ be an expanded $n$-grid where $n$ is a power of five, and let $T$ be a spanning tree of $H$.
It holds that $\LL^*(T,H)$ is at least $\frac{2}{25}n^2\log_5 n$.
\end{lemma}

\begin{proof}
We proceed by induction.
Throughout the proof, we refer to edges not contained in the considered spanning tree as \emph{non-tree edges}.
If $n=1$, the bound holds trivially.
As the base of the induction, we also analyze the case $n=5$ in order to avoid any pathological cases in the induction step.
If $n=5$,
the host grid has at least $2\cdot 5\cdot 4-24=16$ non-tree edges and
so the sum of the perimeters of fundamental cycles associated with the non-tree edges of the host grid
is at least $32$, which is larger than $\frac{2}{25}n^2\log_5 n=2$.

\begin{figure}
\begin{center}
\epsfbox{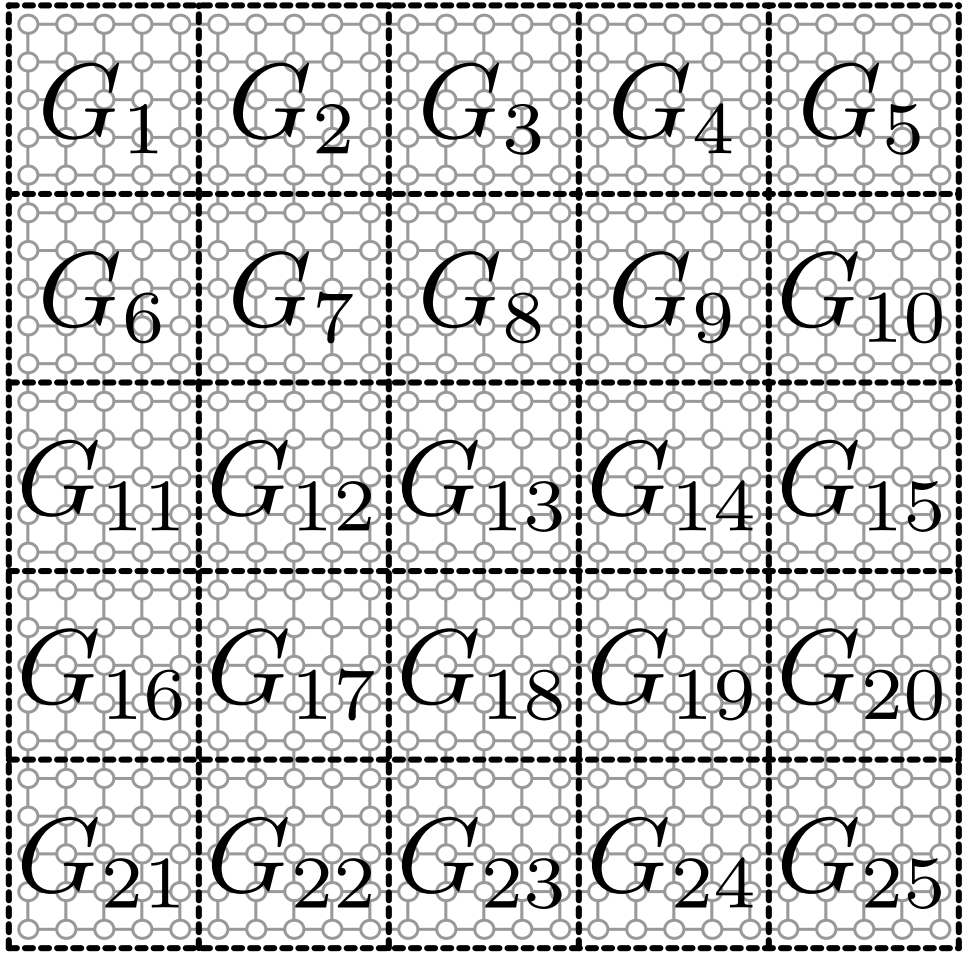}
\end{center}
\caption{The tiling of the host $25$-grid with the $5$-grids $G_1,\ldots,G_{25}$
         as considered in the proof of Theorem~\ref{thm:upper}.}
\label{fig:G25}
\end{figure}

We now present the induction step. Let $n=5^k$ where $k\ge 2$.
Fix an expanded $n$-grid $H$ and a spanning tree $T$ of $H$.
Let $G_1,\ldots,G_{25}$ be the unique collection of $5^{k-1}$-grids that tile the host grid of $H$ and
that are indexed as in Figure~\ref{fig:G25}.
We apply induction to the $25$ expanded $5^{k-1}$-grids $H[G_i/T]$, $i\in [25]$,
together with the spanning trees obtained from $T$ when creating $H[G_i/T]$.
By induction,
the sum of the perimeters of fundamental cycles associated with the non-tree edges of the host grid of $H[G_i/T]$ for each $i\in [25]$
is at least
\begin{equation}
\frac{2}{25} 5^{2(k-1)}\log_5 5^{k-1}=\frac{2}{25} 5^{2(k-1)}(k-1).\label{eq:low1}
\end{equation}
We next show that there are at least $5^{k-1}$ non-tree edges $e$ of the host grid of $H$ such that
\begin{itemize}
\item the edge $e$ is not contained in any $H[G_i/T]$, $i\in [25]$, and
      the perimeter of the fundamental cycle associated with $e$ is at least $2\cdot 5^{k-1}$, or
\item the edge $e$ is contained in $H[G_i/T]$ for some $i\in [25]$ and
      the perimeter of the fundamental cycle associated with $e$ in $H[G_i/T]$
      is smaller by at least $2\cdot 5^{k-1}$ than the perimeter of fundamental cycle associated with $e$ in $H$.
\end{itemize} 

Let $V_0$ be the vertices contained in the subgrid $G_{13}$,
$V_1$ the vertices contained in one of the subgrids $G_7$, $G_8$, $G_9$, $G_{12}$, $G_{14}$, $G_{17}$, $G_{18}$ and $G_{19}$, and
$V_2$ the remaining vertices of the host grid,
where the numbering of the subgrids is as in Figure~\ref{fig:G25}.
Further,
let $C_i$, $i\in\left[5^{k-1}\right]$, be the cycle of the host grid formed
by vertices at distance exactly $i-1$ from a peripheral vertex;
in particular, the cycle $C_1$ is formed by the peripheral vertices.
Note that $C_i$, $i\in\left[5^{k-1}\right]$, are concentric cycles (see Figure~\ref{fig:Ci} for illustration),
which are formed by the vertices of $V_2$.

\begin{figure}
\begin{center}
\epsfbox{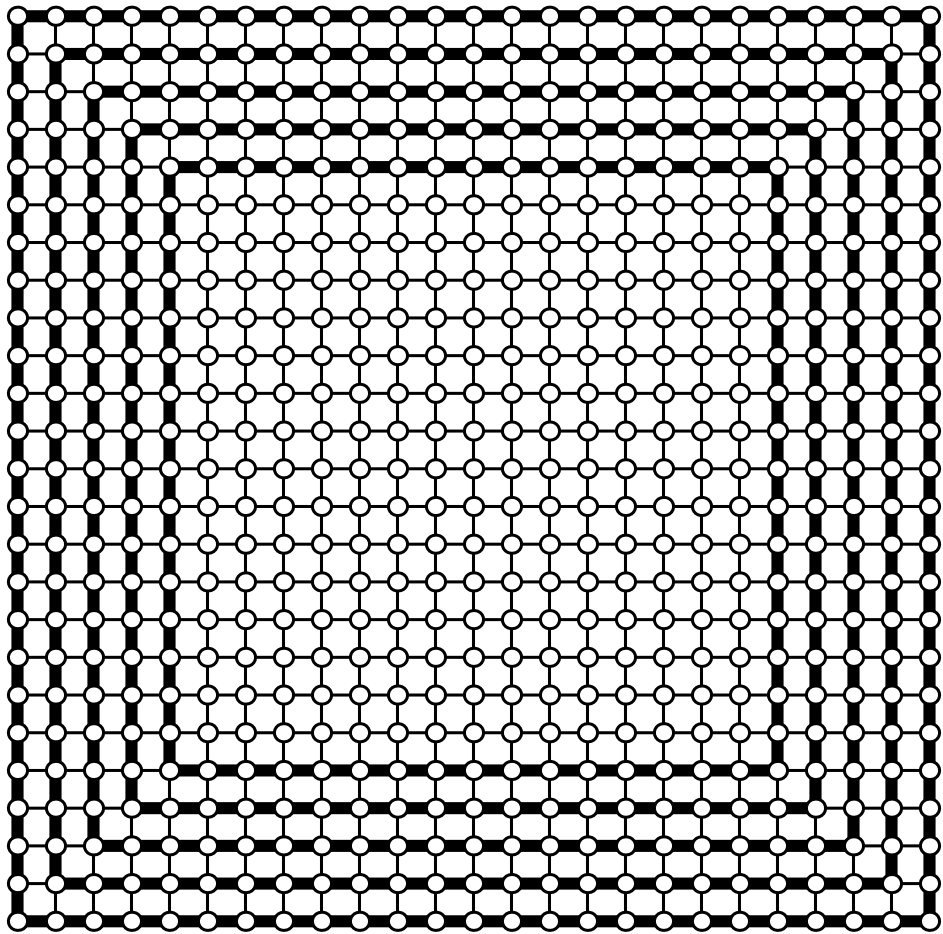}
\end{center}
\caption{The cycles $C_1,\ldots,C_5$ in the host $25$-grid as defined in the proof of Theorem~\ref{thm:lower};
         the cycles are drawn using the bold edges.}
\label{fig:Ci}
\end{figure}

Fix $i\in\left[5^{k-1}\right]$.
For every non-tree edge $e$ of $C_i$, let $P_e$ be the path between the end vertices of $e$ contained in the spanning tree $T$;
note that the edge $e$ and the path $P_e$ form the fundamental cycle associated with $e$.
Let $C'_i$ be the closed walk obtained from $C_i$ by replacing each non-tree edge $e$ with $P_e$.
Recall that the expanded $5^k$-grid $H$ is equipped with an embedding in the plane.
Let $z_0$ be any point of the plane contained inside one of the faces of $G_{13}$ and
consider the plane punctured by removing the point $z_0$.
Since $T$ is a spanning tree and $C'_i$ is a closed walk in $T$,
the closed walk $C'_i$ is contractible in the punctured plane.

We say that a path $P_e$ is \emph{long}
if either the edge $e$ is contained in the first $5^{k-1}$ rows or the last $5^{k-1}$ rows and
the path $P_e$ contains a vertex of a row passing through $V_0$, or
the edge $e$ is contained in the first $5^{k-1}$ columns or the last $5^{k-1}$ columns and
the path $P_e$ contains a vertex of a column passing through $V_0$.
If none of the paths $P_e$ for non-tree edges $e$ of $C_i$ is long,
then the closed walk $C'_i$ would be homotopic to $C_i$,
which is impossible as $C_i$ is not contractible in the punctured plane.
Hence, there exists a non-tree edge $e$ of $C_i$ such that the path $P_e$ is long;
let $e_i$ be any such edge.
If $e_i$ is not contained in any of the grids $G_j$, $j\in [25]$,
then the perimeter of the fundamental cycle associated with $e_i$ is at least $2\cdot 5^{k-1}$, and
the edge $e_i$ has the first property displayed above.
On the other hand,
if $e_i$ is contained in $G_j$ for some $j\in [25]$,
then the perimeter of the fundamental cycle associated with $e_i$ in $H[G_j/T]$
is smaller by at least $2\cdot 5^{k-1}$, and
so the edge $e_i$ has the second property displayed above.
We conclude that there are $5^{k-1}$ non-tree edges, namely, $e_1,\ldots,e_{5^{k-1}}$, such that
each of them has one of the two properties displayed above.

We can now establish the desired lower bound on the sum of the perimeters of fundamental cycles associated with the non-tree edges of the host grid of $H$.
Recall that if an non-tree edge $e$ of the host grid of $H$ is contained in $G_i$,
then the perimeter of the fundamental cycle associated with $e$ in $H$
is at least the perimeter of the fundamental cycle associated with $e$ in $H[G_i/T]$ and
the sum of the perimeters of fundamental cycles associated with such non-edges contained in $G_i$
is at least \eqref{eq:low1}.
Since there are at least $5^{k-1}$ non-tree edges $e$ of the host grid of $H$ such that
either $e$ is not contained in any $H[G_i/T]$, $i\in [25]$, and
the perimeter of the fundamental cycle associated with $e$ is at least $2\cdot 5^{k-1}$, or
the edge $e$ is contained in $H[G_i/T]$ and
the perimeter of the fundamental cycle associated with $e$ got shrunk by at least $2\cdot 5^{k-1}$,
we obtain that $\LL^*(T,H)$ is at least
\[25\cdot\frac{2}{25} 5^{2(k-1)}(k-1)+2\cdot 5^{k-1}\cdot 5^{k-1}=\frac{2}{25}5^{2k}k,\]
which is equal to $\frac{2}{25}n^2\log_5 n$ as desired.
\end{proof}

We are now ready to prove the main result of this section.

\begin{theo}
\label{thm:lower}
Let $H$ be an expanded $n$-grid and $T$ a spanning tree of $H$.
It holds that $\LL^*(T,H)$ is at least $\frac{2}{625}n^2\left\lfloor\log_5 n\right\rfloor$.
\end{theo}

\begin{proof}
Fix an expanded $n$-grid $H$ and a spanning tree $T$ of $H$.
Let $N$ be the largest power of five smaller or equal to $n$, and
let $G$ be any $N$-grid that is a subgrid of the host grid of $H$.
Note that $N\ge n/5$ and $\log_5 N=\left\lfloor\log_5 n\right\rfloor$.

Consider the expanded $N$-grid $H[G/T]$ and
let $T'$ be the spanning tree of $H[G/T]$ resulting from $T$.
By Lemma~\ref{lem:lower}, it holds that
\[\LL^*\left(T',H[G/T]\right)\ge\frac{2}{25}N^2\log_5 N.\]
Since for every non-tree edge $e$ of the host grid of $H[G/T]$,
the perimeter of the fundamental cycle associated with $e$ in $H[G/T]$
is at most the perimeter of the fundamental cycle associated with $e$ in $H$,
we obtain that
\[\LL^*(T,H)\ge \LL^*\left(T',H[G/T]\right)\ge\frac{2}{625}n^2\left\lfloor\log_5 n\right\rfloor.\]
This concludes the proof of the theorem.
\end{proof}

Theorem~\ref{thm:lower} yields the following:

\begin{cor}
\label{cor:lower}
Let $T$ be a spanning tree of an $n$-grid $G_{n,n}$, $n\ge 2$.
The average length of a fundamental cycle with respect to $T$ is at least $\frac{2}{625}\left\lfloor\log_5 n\right\rfloor$.
\end{cor}

\begin{proof}
Fix an $n$-grid $G_{n,n}$ and a spanning tree $T$ of $G_{n,n}$.
By Theorem~\ref{thm:lower}, it holds that
\[\LL^*(T,H)\ge\frac{2}{625}n^2\left\lfloor\log_5 n\right\rfloor.\]
Since the grid $G_{n,n}$ has $2n(n-1)-(n^2-1)=(n-1)^2$ non-tree edges and
the length of a cycle is at least its perimeter,
we obtain that the average length of a fundamental cycle with respect to $T$
is at least
\[\frac{1}{(n-1)^2}\cdot \frac{2}{625}n^2\left\lfloor\log_5 n\right\rfloor\ge\frac{2}{625}\left\lfloor\log_5 n\right\rfloor.\]
\end{proof}

\bibliographystyle{bibstyle}
\bibliography{matgrid}

\end{document}